\documentclass[12pt]{amsart}
\numberwithin{equation}{section}
\newtheorem{theorem}{Theorem}
\newtheorem{lemma}[theorem]{Lemma}
\newtheorem{sublemma}[theorem]{Sublemma}
\newtheorem{corollary}[theorem]{Corollary}
\newtheorem{proposition}[theorem]{Proposition}

\newtheorem{notation}[theorem]{Notation}
\theoremstyle{definition}
\newtheorem{definition}[theorem]{Definition}
\newtheorem{example}[theorem]{Example}

\newtheorem{remark}[theorem]{Remark}

\begin{document}
\title{
Taut submanifolds are algebraic}

\author[Chi]{Quo-Shin Chi}
\thanks{The author was partially supported by NSF Grant
No. DMS-0103838} 
\address{Department of Mathematics \\ Campus Box 1146 \\ Washington
University \\ St. Louis, Missouri 63130} 
\email{chi@math.wustl.edu}

\date{}

\keywords{Dupin hypersurface, taut submanifold, semialgebraic set}
\subjclass[2000]{Primary 53C40}

\begin{abstract} We prove that every (compact) taut submanifold
in Euclidean space is real algebraic, i.e., 
is a
connected component of a real irreducible algebraic variety
in the same ambient space. This answers affirmatively a question of Nicolaas Kuiper raised in the 1980s.
\end{abstract}

\maketitle
\pagestyle{myheadings}
\markboth{QUO-SHIN CHI}{TAUT SUBMANIFOLDS ARE ALGEBRAIC}

\section{Introduction} 
An embedding $f$ of a compact, connected manifold $M$ into Euclidean space
${\mathbb R}^n$ is
\textit{taut} if every nondegenerate (Morse) Euclidean distance function, 
\[
L_p:M \to {\mathbb R}, \quad L_p(z) = d(f(z),p)^2, \quad p \in {\mathbb R}^n,
\]
has $\beta (M,{\mathbb Z}_2)$ critical points on $M$, where
$\beta (M,{\mathbb Z}_2)$ is the sum of the ${\mathbb Z}_2$-Betti numbers
of $M$.  That is, $L_p$ is a perfect Morse function on $M$.

A slight variation of Kuiper's observation in~\cite{Ku1} 
gives that tautness can be rephrased by the property that
\begin{equation}\label{d}
H_j(M\cap B,{\mathbb Z}_2)\rightarrow H_j(M,{\mathbb Z}_2)
\end{equation}
is injective for all closed disks $B\subset{\mathbb R}^n$ and all
$0\leq j\leq\dim(M)$. As a result, tautness is a conformal invariant,
so that via stereographic projection we can reformulate the notion of tautness in the sphere $S^n$ using
the spherical distance functions. Another immediate consequence
is that if $B_1\subset B_2$, then
\begin{equation}\label{inj}
H_j(M\cap B_1)\rightarrow H_j(M\cap B_2)
\end{equation}
is injective for all $j$.

Kuiper in~\cite{Ku} raised the question whether all taut submanifolds in ${\mathbb R}^n$ are real algebraic.
We established in~\cite{CCJ} that a taut submanifold
in ${\mathbb R}^n$ is real algebraic in the sense that, it is a
connected component of a real irreducible algebraic variety
in the same ambient space, provided the submanifold is of dimension no greater than 4.

In this paper, we prove that all taut submanifolds in ${\mathbb R}^n$ are real
algebraic in the above sense, so that each is a connected component of a real irreducible algebraic variety in the same ambient space. 
In particular, any taut hypersurface in ${\mathbb R}^n$ is described as
$p(t)=0$ by a single irreducible polynomial $p(t)$ over ${\mathbb R}^n$.
Moreover, since a tube with a small radius of a
taut submanifold in ${\mathbb R}^n$ is a taut hypersurface~\cite{Pi}, which recovers
the taut submanifold along its normals, 
understanding a taut submanifold, in principle, comes down to understanding
the hypersurface case defined by a single algebraic equation.

It is more convenient to prove that a taut
submanifold in the sphere is real algebraic, though occasionally we will switch
back to Euclidean space when it is more convenient for the argument.
Since a spherical distance function $d_p (q) = \cos^{-1} (p \cdot q)$ has the same critical points as
the Euclidean height function
$\ell_p (q) = p \cdot q$, for $p,q \in S^n$, a compact submanifold $M \subset S^n$ is taut if and only if it is
{\em tight}, i.e., every nondegenerate height function $\ell_p$ has the
total Betti number $\beta (M, {\mathbb Z}_2)$ of critical points
on $M$.
We will use both $d_p$ and $\ell_p$ interchangeably, whichever is
more convenient for our argument.

Our proof is based on the {\em local finiteness property}~\cite[Definition 7]{CCJ}
that played the decisive role for a taut submanifold to be algebraic when its dimension is $\leq 4$. This property
is parallel in spirit to the Riemann extension theorem in complex variables.
Namely, let ${\mathcal G}$ (i.e., the first letter of the word "good")
be the subset of a taut hypersurface $M$ where the 
principal multiplicities are locally constant, and let ${\mathcal G}^c$ be its complement in $M$. Let $S\subset{\mathcal G}^c$ be closed in $M$; $S$
is typically a set difficult to manage. 
If $M\setminus S$ is well-behaved to have finitely many connected components,
and moreover, each $x\in {\mathcal G}^c\setminus S$
has a small neighborhood $U$ in $M$ for which ${\mathcal G}\cap U$ is 
also well-behaved to have finitely many connected components, then $M$ is algebraic.

As is pointed out and manifested in~\cite{TT}, focal sets play an important role in the study of submanifolds. In general, however, they are nonsmooth, where, for instance, the best one can expect of that of a 
hypersurface in $S^n$ is that
it is at least of Hausdorff codimension 2~\cite{CCJ}. One thus expects that the (unit) tangent cones of a focal set can be a useful tool for understanding such nonsmooth 
objects, though in general the tangent cones of a focal set themselves are also rather untamed. 
In the case of taut submanifolds, nonetheless, we can tame a focal set when we use the mathematical induction 
on the dimension of the ambient sphere $S^n$ for which all taut submanifolds are algebraic. Since a curvature surface $Z$, 
which is taut by Ozawa theorem~\cite{Oz}, lies in a curvature sphere, the induction 
hypothesis implies that $Z$ is algebraic, which leads us to the free access of parametrizing the focal set of $Z$ and
its unit tangent cone at a point by semialgebraic sets. (For the reader's convenience, we include a section on semialgebraic sets and some of their important properties.) 
From this we can show, by conducting certain dimension estimates, facilitated and made precise by the introduction of unit tangent cones, that barring 
a closed set $S\subset{\mathcal G}^c$ 
of Hausdorff $(\dim(M)-1)$-measure zero that, hence by~\cite{SY}, does not disconnect the taut hyperusrface $M\subset S^n$, the set 
${\mathcal G}^c$ is essentially a manifold of codimension 1 in $M$, as expected. The local finiteness property then results in the hypersurface case,
whence follows the algebraicity of a taut submanifold.



\section{preliminaries}

\subsection{The Ozawa theorem} A fundamental result on taut submanifolds is due to Ozawa~\cite{Oz} (see also~\cite{TT1} for its generalization to the Riemannian case).

\begin{theorem}[Ozawa]
\label{Ozawa}
Let $M$ be a taut submanifold in $S^n$, and let $\ell_p, p \in S^n$, be a linear height
function on $M$.  Let $x \in M$ be a critical point of $\ell_p$, and let $Z$ be the connected component of the
critical set of $\ell_p$ that contains $x$. Then $Z$ is \\
{\rm (a)} a smooth compact manifold of dimension equal to the nullity of the Hessian of $\ell_p$ at $x$;\\
{\rm (b)} nondegenerate as a critical manifold;\\
{\rm (c)} taut in $S^n$.
\end{theorem}
In particular, $\ell_p$ is perfect Morse-Bott~\cite{Bt}. We call such a connected component of
a critical set of $\ell_p$ a {\em critical submanifold} of $\ell_p$. 

An important consequence of Ozawa's theorem is the following~\cite{Ce}.

\begin{corollary}\label{cor}
Let $M$ be a taut submanifold in $S^n$.
Then given any principal space
$T$ of any shape operator $S_{\zeta}$ at any point $x\in M$, there exists a
submanifold $Z$ (called a {curvature surface}) through $x$ whose tangent space at $x$ is $T$. That is,
$M$ is Dupin~\cite{Ma},~\cite{Pi}.
\end{corollary}

Let us remark on a few important points in the corollary. It is convenient
to work in the ambient Euclidean space ${\mathbb R}^n$. Let $\mu$ be the
principal value associated with $T$. Consider the focal point $p=x+\zeta/\mu$.
Then the critical submanifold $Z$ of the (Euclidean) distance function $L_p$
through $x$ is exactly the desired curvature surface through $x$. The unit
vector field
\begin{equation}\label{focal}
\zeta(y):=\mu(p-y)
\end{equation}
for $y\in Z$ extends $\zeta$ at $x$
and is normal to and parallel along $Z$. 
The $(n-1)$-sphere of radius $1/\mu$
centered at $p$ is called the {\em curvature sphere} of $Z$. In particular, two different focal points
cannot have the same critical submanifold.


\subsection{A brief review on semialgebraic sets}

A \textit{semialgebraic subset} of ${\mathbb R}^{n}$ is one which is a finite
union 
of sets of the form
$$
\cap_{j} \{x\in {\mathbb R}^{n}: F_j(x) * 0\},
$$
where * is either $<$ or $=$, $F_j\in {\mathbb R}[X_1,\dots,X_n]$, the polynomial ring in $X_1,\cdots,X_n$, and the
intersection is finite. An {\em algebraic subset} is one when $*$ is $=$
for all $j$ without taking the finite union operation. Clearly, an algebraic set is semialgebraic.

It follows from the definition that any finite union or intersection of semialgebraic sets is semialgebraic,
the complement of a semialgebraic
set is 
semialgebraic, and hence a semialgebraic set taking away another
semialgebraic set leaves a semialgebraic set.
Moreover, the projection $\pi:{\mathbb R}^n\rightarrow {\mathbb R}^k$ sending
$x\in{\mathbb R}^{n}$ to its first $k$ coordinates maps a
semialgebraic set to a 
semialgebraic set.
In particular, the topological closure and interior of a semialgebraic set are
semialgebraic. 

\begin{example}\label{RNK} Let $L$ be a $k$-by-$l$ matrix each of whose entries is a polynomial over ${\mathbb R}^s$.
Then the subset of ${\mathbb R}^s$ where $L$ is of rank $t$ is semialgebraic.
\end{example}

\begin{proof} Let the $j$-by-$j$ minors of $L$ be $F_{j,1},\cdots,F_{j,i_j},1\leq j\leq\min(k,l)$. The set $R_j\subset {\mathbb R}^s$ where $L$ is of rank 
$\leq j-1$ is given by 
setting $F_{j,1}=\cdots=F_{j,i_j}=0$, which is an algebraic set. The subset of ${\mathbb R}^s$ for which $L$ is of rank $t$ is the semialgebraic set
$R_{t+1}\setminus R_t$.
\end{proof}

\begin{example}\label{ex} Let $A\subset {\mathbb R}^m$ be a
semialgebraic set. Consider the set
$$
B:=\{(x,y)\in (A\times A):x\neq y\}.
$$
$B$ is semialgebraic since it is the complement of the diagonal
of $A$ in $A\times A$. Consider the set
$$
C:=\{(x,y,z)\in B\times S^{m-1}:(x,y)\in B, z=(x-y)/|x-y|\}.
$$
$C$ is semialgebraic since it is defined by the polynomial equations
$$
(z_i)^2|x-y|^2=(x_i-y_i)^2
$$
where $x_i,y_i,z_i$ are the coordinates of $x,y,z$. Let $D$ be the
topological closure of $C$
in ${\mathbb R}^m\times {\mathbb R}^m\times S^{m-1}$, and let
$$
proj:{\mathbb R}^m\times {\mathbb R}^m\times S^{m-1}\rightarrow {\mathbb R}^m \times {\mathbb R}^m
$$
be the standard projection. The preimage
$$
E:=proj^{-1}((p,p)),\quad p\in A
$$
is also semialgebraic.

In fact, $E$ is obtained by taking the limits of converging subsequences
of $(q_n-p)/|q_n-p|$ for all converging sequences $(q_n)$ to $p$.
\end{example}

\begin{definition}\label{cone} We call the set $E$ in Example~\ref{ex} the {\em unit tangent cone} of the set
$A$ at $p$, remarking that the process of defining the unit tangent cone can be done
for any set. We say a sequence $(q_n)$ in $A$ {\em converges to a unit tangent cone
vector $e$ at $p$} if $q_n$ converges to $p$ and $(q_n-p)/|q_n-p|$ converges to $e$ .
\end{definition}

\begin{definition}\label{conenhd} For $e$ in Definition~\ref{cone} and any $\delta>0$, we let ${\mathcal O}_e$ be the semialgebraic open set 
$$
{\mathcal O}_e(\delta):=\{q\in{\mathbb R}^m:|e-(q-p)/|q-p||<\delta.
$$  
\end{definition}
${\mathcal O}_e(\delta)$ is a generalized cone with vertex $p$ and axis $e$.

 \vspace{2mm}

A map $f:S\subset {\mathbb R}^{n}\rightarrow{\mathbb R}^{k}$ over a semialgebraic $S$
is \textit{semialgebraic} if its graph in ${\mathbb R}^{n}\times{\mathbb R}^{k}$
is a semialgebraic set. 
It follows that the image of a semialgebraic map
$f:S\subset {\mathbb R}^{n}\rightarrow{\mathbb R}^{k}$
is semialgebraic, via the composition
${\rm graph}(f)\subset{\mathbb R}^{n}\times{\mathbb R}^{k}\rightarrow{\mathbb R}^{k},$
where the last map is the projection onto the second summand.

\begin{example}\label{rnk} The conclusion in Example~\ref{RNK} continues to hold if the entries of the 
matrix involved consist of semialgebraic functions.
\end{example}

A \textit{Nash function} is a $C^{\infty}$ semialgebraic map from an open
semialgebraic subset of ${\mathbb R}^{n}$ to ${\mathbb R}$.
A real analytic function $f$ defined
on an open 
semialgebraic subset $U$ of ${\mathbb R}^n$ is \textit{analytic algebraic} if it
is a solution of a polynomial equation on $U$ of the form,
\begin{equation}\label{Eq2.9}
a_{0}(x) f^s(x) + a_{1}(x) f^{s-1}(x) + \cdots + a_{s}(x) = 0,
\end{equation}
where $a_0(x)\neq 0,a_{1}(x), \cdots, a_s(x)$ are polynomials over
${\mathbb R}^{n}$.
These two concepts are in fact equivalent\cite[p.\ 165]{BCR}, that  
a function is Nash if and
only if it is analytic algebraic.

The following example is instructive.
\begin{example}\label{Ex2.1} For any number $\epsilon$ satisfying
$0 < \epsilon < 1$, the open disk
\begin{equation}\label{Eq2.10}
B^n(\epsilon) = \{s = (s^1,\dots,s^n)\in {\mathbb R}^n : |s| < \epsilon \}
\end{equation}
is an open semi-algebraic subset of ${\mathbb R}^n$.  The function
\begin{equation}\label{Eq2.11}
s^0 = \sqrt{1-|s|^2}
\end{equation}
on $B^n(\epsilon)$ is analytic algebraic, since $(s^0(x))^2 + a_0(x) =
0$ on $B^n(\epsilon)$, where $a_0(x)$ is the polynomial $|s|^2-1$ on
${\mathbb R}^n$. Partial derivatives of all orders of $s^0$ are analytic
algebraic.  In fact, an elementary calculation and induction argument
shows that if $D_i$ denotes the partial derivative with respect to
$s^i$, then
\[
D_{i_1\dots i_k} s^0 = \frac{a_k(s)}{(s^0)^m}
\]
where $a_k(s)$ is a polynomial on ${\mathbb R}^n$ and $m$ is a positive
integer.  Therefore, 
\[
(s^0)^{2m} (D_{i_1\dots i_k} s^0)^2-a_k(s)^2 = 0
\]
is an equation of the form~(\ref{Eq2.9}), since $(s^0)^2$ is a
polynomial on ${\mathbb R}^n$.
\end{example}
A slight generalization of the single-variable case in~\cite[p.\ 54]{BCR}, shows that the partial derivatives of any Nash
function are again Nash functions.  

Let $S$ be a semialgebraic subset of ${\mathbb R}^{n}$. The
\textit{dimension} of $S$, denoted $\dim (S)$, is the dimension of the ring
$R = {\mathbb R}[x^1,\cdots,x^n]/\mathcal I(S)$, where ${\mathcal I}(S)$ the ideal of all polynomials vanishing
on $S$, which is
the maximal length of chains of prime ideals of $R$.
As usual, it is proved that
if $S$ is a semialgebraic subset of ${\mathbb R}^n$ that is a $C^\infty$
submanifold of ${\mathbb R}^n$ of dimension $d$, then $\dim (S) = d$.

A semialgebraic subset $M$ of ${\mathbb R}^{m}$ is a \textit{Nash submanifold} of
${\mathbb R}^{m}$ of dimension 
$n$ if for every point $p$ of $M$, there exists a Nash diffeomorphism
$\psi$ from an open semialgebraic  
neighborhood $U$ of the origin in ${\mathbb R}^{m}$ into an open
semialgebraic neighborhood $V$ 
of $p$ in ${\mathbb R}^{m}$ such that $\psi(0)=p$ and
$\psi(({\mathbb R}^{n}\times\{0\})\cap U)= M\cap V$. Here, by a \textit{Nash  
diffeomorphism} $\psi$ we mean the coordinate
functions of $\psi$ and $\psi^{-1}$ are Nash functions.

Let $M$ be a Nash submanifold of
${\mathbb R}^{m}$. A mapping $f:M\rightarrow{\mathbb R}$ is a  
\textit{Nash mapping} if it is semialgebraic, and for every $\psi$ in the
preceding definition, $f\circ\psi|_{{\mathbb R}^{n}\cap U}$ 
is a Nash function.

As in the $C^{\infty}$ case, the semialgebraic version of the
inverse and implicit function theorems also hold~\cite[p.\ 56]{BCR}.
Moreover, the semialgebraic version of the (Nash) tubular neighborhood
theorem over Nash manifolds is true~\cite[p.\ 199]{BCR}.

Of special importance to us is the slicing theorem~\cite[p.\ 30]{BCR},
for which we only give the special version we need for the sake of clarity.

\begin{theorem}\label{slice} (Slicing Theorem) Let
$$
p_j(z,\lambda):=\lambda^{s_j}+a_{s_j-1}^j(z)\lambda^{s_j-1}+\cdots+a_1^j(z)\lambda+a_0^j(z),\quad 1\leq j\leq a,
$$
be real polynomials in $m+1$ variables $(z,\lambda)\in{\mathbb R}^m\times{\mathbb R}$ with
degree $s_j$ in $\lambda$. Then there
is a partition of\; ${\mathbb R}^m$ into a finite number of (disjoint)
semialgebraic sets
$A_1,\cdots,A_l$, and for each $i=1,\cdots,l$, a finite number (possibly zero)
of semialgebraic functions
$$
\zeta_{i,1}<\cdots<\zeta_{i,s_i}:A_i\rightarrow{\mathbb R},
$$
such that for every $z\in A_i$, $\lambda=\zeta_{i,1},\cdots,\zeta_{i,s_i}$
are the distinct roots of $p_j(z,\lambda)=0,1\leq j\leq a$.
\end{theorem}
The number of these semialgebraic functions may be zero because a real polynomial
may have no real roots.
More importantly, the slicing theorem encodes the multiplicities
of the roots into account. To see this for our later application, we start with a single polynomial
$$
p(z,\lambda):=\lambda^{s}+a_{s-1}(z)\lambda^{s-1}+\cdots+a_1(z)\lambda+a_0(z).
$$
The slicing theorem provides us with root functions $\zeta_{i,1},\cdots,\zeta_{i,s_i}$ over $A_i,1\leq i\leq l$. The polynomial
$p(z,\lambda)$ in the variable $\lambda$ has repeated roots if and only if
$s_i<s$,
in which case the largest $j+1$ for which 
$$
\partial^j p/\partial \lambda^j=0
$$
evaluated at $\zeta_{i,1}(z),\cdots,\zeta_{i,s_i}(z)$ 
is the multiplicity of the respective root.

Therefore, to find the root functions of $p(z,\lambda)$ with multiplicities, we solve, for each $k$,
$$
 \partial^j p/\partial \lambda^j=0,\quad 0\leq j\leq k,
$$
with the solution sets $A_{k,1},\cdots,A_{k,l_k}$ and root functions
$$
\xi_{k,i,1}<\cdots<\xi_{k,i,s_{k,i}}:A_{k,i}\rightarrow{\mathbb R},\quad 1\leq i\leq l_k.
$$
Employ the two operations of taking the complement and finite intersection of sets to perform on $A_{k,1},\cdots,A_{k,l_k}$ for $k=0,\cdots,s$, 
knowing $A\cup B=(A^c\cap B^c)^c$ and $A\setminus B=A\cap B^c.$ For instance, 
$$
\cup_iA_{0,i}\setminus\cup_i A_{1,i}
$$
is the semialgebraic subset where $p(z,\lambda)$ has only simple roots with the root functions $\xi_{0,i,j}$ that carry multiplicity 1.
Similarly,
$$
((\cup_i A_{1,i})\cap(\cup_i A_{2,i}))\setminus\cup_i A_{3,i}
$$
is the semialgebraic subset where $p(z,\lambda)$ has only double and triple roots, with the root functions 
$\xi_{k,i,j},k=1,2,$ that carry multiplicity $k+1$, etc.
Eventually, we end up with semialgebraic sets $V_1,\cdots,V_\tau$ that give rise to the root functions, with
multiplicities.
The set $T$ where the closures of $V_1,\cdots,V_\tau$ intersect, which
is also semialgebraic, is where the multiplicities of the roots of $p(z,\lambda)$
are not locally constant. It follows that $\dim(T)<k$, the ambient Euclidean dimension, because of the following dimension property~\cite[p.\ 53]{BCR}.

\begin{proposition}\label{<} Let $X\subset {\mathbb R}^s$ be a semialgebraic set and let $\overline{X}$ be its topological closure. 
Then $\dim({\overline X}\setminus X)<\dim(X)$.
\end{proposition}

A weaker smooth version~\cite{Re} than $\dim(T)<k$ above says that the set where
the principal multiplicities of the shape operator $S_\xi$ is not locally
constant,
as $\xi$ varies on the unit normal bundle of a smooth submanifold, is
nowhere dense.

Lastly, we record the important decomposition property~\cite[p.\ 57]{BCR}.

\begin{proposition}\label{tri} A semialgebraic set is the disjoint union of a finite number of Nash submanifolds
$N_j$, each Nash diffeomorphic to an open cube $(0,1)^{\dim(N_j)}$.
\end{proposition}
We call these Nash submanifolds in the decomposition {\em open cells} henceforth.

\subsection{The local finiteness property}

Recall the local finiteness
property in~\cite{CCJ} that holds the key for proving that a taut hypersurface $M$ is
real algebraic when its dimension is $\leq 4$. 
We denote by ${\mathcal G}$ the subset of $M$ where
the multiplicities of principal curvatures are locally constant, and by
${\mathcal G}^c$ the complement of ${\mathcal G}$ in $M$.

\begin{definition}\label{finite} A connected Dupin hypersurface $M$ of
$S^n$ has
the \textit{local finiteness property} if there is a set $S\subset {\mathcal G}^c$,
closed in $M$, such that $S$ disconnects $M$ in finitely many connected
components, and 
for each point $x\in {\mathcal G}^c\setminus S$
there is an open neighborhood $U$ of $x$ in $M$
such that ${\mathcal G}\cap U$ contains finitely many
connected open sets (whose union is dense in $U$).
\end{definition}

The technical advantage of excising a set $S$ in the definition is that, we can remove 
certain types of points whose principal multiplicities are not locally constant and hard 
to handle without affecting establishing the algebraicity of the taut submanifold, as will become evident later.

\section{Some local analysis and its implication} We first handle the case when $M$ is a hypersurface.
Fix a unit normal
field $\bf n$ over $M$ once and for all.
We label the principal curvatures
of $M$ by 
$\lambda_1\leq\cdots\leq\lambda_{n-1}$, which are Lipschitz-continuous functions on $M$
because the principal curvature functions on the linear space ${\mathcal L}$ of all symmetric matrices
are Lipschitz-continuous by general matrix theory~\cite[p.\ 64]{Ba},
and the Hessian of of $M$ is a smooth function from $M$ into ${\mathcal L}$.
Let $\lambda_j=\cot(t_j)$ for $0<t_j<\pi$. We have the Lipschitz-continuous focal maps
\begin{equation}\label{fmap}
f_j(x)=\cos(t_j)x+\sin(t_j){\bf n}.
\end{equation}
In fact, the $l$-th focal point $f_l(x)$, counting multiplicities, along $\bf n$ emanating from $x$ is
antipodally symmetric to the $(n-l)$-th focal point, counting multiplicities, along $-{\bf n}$ emanating from
$x$. The spherical distance functions $d_{f_l(x)}$ tracing backward
following $-{\bf n}$ thus
assumes the same critical point $x$ as the distance function $d_{-f_l(x)}$
tracing backward following $\bf n$; thus we may just consider the former case
without loss of generality. Accordingly, henceforth we refer to a focal point $p$
as being $f_j(x)$ for some $x$ and $j$ with an obvious modification if necessary.

\begin{notation} Let $p$ be the focal point of a fixed curvature surface $Z$
and let $q\neq p$ be the focal point of a nearby curvature surface. Set
$$
g:=\ell_q-\ell_p
$$
considered as the perturbation of the height function $\ell_p$ to the
nearby $\ell_q$ by $g$.
\end{notation}

We assume
$$
\ell_p(Z)=0
$$
without loss of generality.
Let $W\subset M$ be a
tubular neighborhood of $Z$ so small that $Z$ is the only critical
submanifold of $\ell_p$ in $W$ and let
\begin{equation}\label{pI}
\pi:W\rightarrow Z
\end{equation}
be the projection.

\begin{definition} Notation as above, we will call such a $W$ a {\em neck}\;
around $Z$.
\end{definition}

Let us parametrize $W$ by
$v_1,\cdots,v_s,u_1,\cdots,u_{n-1-s}$ with $u_1=\cdots=u_{n-1-s}=0$ parametrizing $Z$
such that
\begin{equation}\label{pII}
\pi:(v_1,\cdots,v_s,u_1,\cdots,u_{n-1-s})\mapsto (v_1,\cdots,v_s).
\end{equation}
By a linear change of $u_1,\cdots,u_{n-1-s}$, we may assume, around $0\in Z$,

\begin{eqnarray}\label{el}
\aligned
\ell_p&=\sum_{j=1}^{n-1-s}\alpha_ju_j^2+O(3),\\
g&=h(u)+k(v)+\sum_{jk}\beta_{jk}v_jv_k+\sum_{jk}\gamma_{jk}u_jv_k+O(3),
\endaligned
\end{eqnarray}
with
$$
h(u)=\sum_i a_iu_i+\sum_{j,k}b_{jk}u_ju_k,\quad k(v)=\sum_l c_l v_l
$$
for some small coefficients $a_i,\beta_{jk},\gamma_{jk},b_{jk},$ and $c_l$ all in the magnitude
of $|q-p|$ by the linear nature of $g$, where $\alpha_j$ are fixed nonzero constants.
We obtain
$$
F_j:={\partial(\ell_p+g)}/{\partial u_j}=a_j+2\alpha_ju_j+2\sum_l
b_{jl}u_l+\sum_{k}\gamma_{jk}v_k+O(2)
$$
for $1\leq j\leq n-1-s$, and 
$$
G_i:={\partial(\ell_p+g)}/{\partial v_i}=c_i+2\sum_k \beta_{ik}v_k+\sum_{j}\gamma_{ji}u_j+O(2)
$$
for $1\leq i\leq s$. We calculate 

\begin{equation}\label{implicit}
\partial(F_1,\cdots,F_{n-1-s},G_1,\cdots,G_s)/\partial(u_1,\cdots,u_{n-1-s},v_1,\cdots,v_s)
=\begin{pmatrix}
\Theta&\gamma\\\gamma^{tr}&\Gamma\end{pmatrix},
\end{equation}
where
$$
\Theta:=\begin{pmatrix}2\alpha_j\delta_{jl}+2b_{jl}\end{pmatrix},\quad \gamma:=\begin{pmatrix}\gamma_{jk}\end{pmatrix},\quad
\Gamma:=\begin{pmatrix}\beta_{ik}\end{pmatrix},
$$
at $u_1=\cdots=u_{n-1-s}=v_1=\cdots=v_s=0$. Note that $\Theta$ is nonsingular
when we let $|q-p|$ be much smaller than
$\min_j|\alpha_j|$. 

Due to the compactness of $Z$, we can cover $W$ by a finite number of coordinate charts as above, so that
there is an open disk ${\mathcal D}$ centered at $p$ whose radius is so small that $\Theta$ is nonsingular
in any of these coordinate charts for all $q\in{\mathcal D}$.
In other words, since the left hand side of~\eqref{implicit} is the Hessian of $\ell_q=\ell_p+g$, 
we may assume without loss of generality that
the Hessian of $\ell_q$ is of rank $\geq n-1-s$ in $W$
whenever $q\in{\mathcal D}$.

Recall Definition~\ref{cone}. Let us be given a vector $e$, which will in future applications be 
a unit tangent cone vector at $p$ of an appropriately chosen set. Assume that $Z$ 
is not on a level set
of $\ell_e$.
Since $Z$ is taut, the height function $\ell_e$ cuts $Z$ in
several critical submanifolds $Z_1,\cdots, Z_m$. Let us
consider $Z_1$, for instance. Assume the codimension of $Z_1$ in $Z$ is $t$.

Let $T_1$ be a neck of $Z_1$ in $Z$ and let
$$
N_1:=\pi^{-1}(T_1)\subset W
$$ 
with $\pi$ given in~\eqref{pI}.
Retaining the notation in~\eqref{pII}, let us parametrize $N_1$ by the variables
$$
v_1,\cdots,v_t,v_{t+1},\cdots,v_s, u_1,\cdots,u_{n-1-s}
$$ 
around 0, where
$v_{t+1}\cdots,v_s$ parametrize $Z_1$ and $v_1,\cdots,v_s$ parametrize $T_1$
around $Z_1$ in $Z$.
It is understood that 0 in the coordinate system corresponds to a point on
$Z_1$. 

Note that $\ell_e$ now assumes a simpler form
\begin{equation}\label{ell}
\ell_e=h(u)+\sum_{j=1}^t\beta_{j}v_j^2+\sum_{jk}\gamma_{jk}u_jv_k+O(3),
\end{equation}
where all $\beta_i$ are nonzero since $Z_1$
is a critical submanifold of $\ell_e|_Z$ in $Z$.

By continuity, there is a small open set ${\mathcal O}_e$ given in Definition~\ref{conenhd} (we ignore the radius $\delta$ in the definition) 
such that whenever $q\in{\mathcal O}_e$ the unit vector
$$
e_q:=(q-p)/|q-p|
$$
with 
\begin{equation}\label{elll}
\ell_{e_q}= h^q(u)+ k^q(v)+\sum_{jk}\beta_{jk}^{e_q}v_jv_k+\sum_{jk}\gamma_{jk}^{e_q}u_jv_k+O(3),
\end{equation}
where $k^q(v)$ is linear in $v$, satisfies that the upper $t$-by-$t$ block of the
matrix $\begin{pmatrix}\beta_{jk}^{e_q}\end{pmatrix}$ is nonsingular,
or equivalently, $\begin{pmatrix}\beta_{jk}^{e_q}\end{pmatrix}$ is of rank $\geq t$
whenever $q\in{\mathcal O}_e$ (by shrinking the coordinates if necessary). It follows that, when we substitute 
\begin{equation}\label{gl}
g_q:=|q-p|\ell_{e_q},\quad l=1,2,\cdots,
\end{equation}
into~\eqref{implicit}, with 
\begin{equation}\label{sub1}
\ell_{q}=\ell_p+g_q,
\end{equation}
we see the Hessian of
$\ell_{q}$ are all of rank $\geq n-1-s+t$, or all of kernel dimension
$\leq \dim(Z_1)$, whenever $q\in{\mathcal O}_e$. That is,
\begin{equation}\label{C}
\dim(C_q)\leq \dim(Z_1)
\end{equation} 
whenever a critical submanifold $C_q$ of $\ell_{q}$ lies in $N_1$. Similarly, the same holds for other $N_j$ as well.

On the other hand, at a point $x\in Z$ away from $Z_1,\cdots,Z_m$,
we can still parametrize $W$ around $x$ by $v_1,\cdots,v_s,u_1,\cdots,u_{n-1-s},$
where $v_1,\cdots,v_s$ parametrize $Z$ around $x$ identified with 0. Then slightly differently from the
earlier expression we have

\begin{eqnarray}\nonumber
\aligned
\ell_p&=\sum_{j=1}^{n-1-s}\alpha_ju_j^2+O(3),\\
\ell_e&=h(u)+\sum_{i=1}^s\gamma_iv_i+\sum_{i=1}^s\delta_{i}v_i^2+\sum_{jk}\gamma_{jk}u_jv_k+O(3),
\endaligned
\end{eqnarray}
where at least one of $\gamma_i$ is nonzero since $x$ is not a
critical point of
$\ell_e$ on $Z$.
As a consequence, following the argument below~\eqref{ell} we conclude
that the gradient of $\ell_q$ is nonzero in a neighborhood
of $x$ in $W$, so that 
no critical submanifold $C_{e_q}$ of $\ell_q$ passes through this
neighborhood for $q\in{\mathcal O}_e$.
In conclusion, we have the following.

\begin{proposition}\label{P} Notations and conditions as above, let $Z$ be a curvature
surface
with focal point $p$, and let $e$ be a vector, which will be a unit tangent cone vector at p of an appropriately chosen set later,
such that $Z$ is not on any level set of $\ell_e$. Let
$Z_1,\cdots,Z_m$ be the critical submanifolds of $\ell_e|_Z$ in $Z$ with small disjoint necks
$T_j$ of $Z_j$ in $Z$ and $N_j$ of $Z_j$ in $W, j=1,\cdots,m$.  
Then there is an open set ${\mathcal O}_e$ 
such that for every $q\in{\mathcal O}_e$ 
a focal submanifold $C_q$ of $\ell_q$ in $W$ is contained in a unique $N_s$ for some $s\leq m$ with
$$
\dim(C_q)\leq\dim(Z_s).
$$
\end{proposition}

\begin{proof} Since $Z_1,\cdots Z_m$ are disjoint in $Z$, as above we cover them by disjoint open
necks $T_j$ in $Z$ and $N_j\supset T_j$ in $W,1\leq j\leq m$;
then cover the
complement of $\cup_{j=1}^m T_j$ in $Z$ by finitely many
small open balls $B_j$ in $Z$ and open balls $O_j\supset B_j$ in $W, 1\leq j\leq a$, such that
no critical submanifold passes through $O_i,\forall i$. Therefore, $C_q\subset N_s$ for some unique $s\leq m$. 

The second statement is~\eqref{C}.
\end{proof}





\begin{corollary}\label{CCORO} Suppose all unit tangent cone vectors $e$ at $p$ of focal points $q$ near $p$ satisfy that $Z$ is not on any level set of $\ell_e$. 
Assume in any 
open ball $O_x(1/j)$ of radius $1/j$ centered at $x\in Z$, there is a point $y_j$ at which the principal multiplicities are not locally constant, or equivalently,
there is a curvature surface $C_j$ through $y_j$ whose dimension is not locally constant (so that it is not diffeomorphic to a sphere). Let $e$ be the unit tangent cone vector 
at $p$ of the focal points $q_j$ of $C_j$ and let $Z_1,\cdots,Z_m$ be the focal submanifolds of $\ell_e|_Z$ in $Z$. Then there is a unique $Z_s$ through $x$ that is 
not diffeomorphic to a 
sphere, or equivalently, whose dimension in $Z$ is not locally constant. 
\end{corollary}      

\begin{proof} Notations are as in Proposition~\ref{P}. That $x\in Z_s$ results when we let $j$ get larger and larger while shrinking the necks $N_1,\cdots,N_m$
more and more. 

We show $Z_s$ is not diffeomorphic to a sphere. For a sufficiently
large $j$, since $C_j$ lies in the neck $N_s$ of $Z_s$, tautness of $C_j$ and $Z_s$ imply that the topology of $C_j$ embeds in the topology of $Z_s$ (see the remark below). Therefore, 
$Z_s$ cannot be a sphere. Otherwise $C_j$ would be a sphere, which is not the case.
\end{proof}

\begin{remark}\label{importantremark} 
We let
$$
M_{p,\epsilon}:= \ell_p^{-1}((-\infty,\epsilon]), \quad 
M_{p,\epsilon}^\circ:= \ell_p^{-1}((-\infty,\epsilon)). 
$$
Fix noncritical values $a,b$ of $\ell_p$ with $a<0<b$ and 0 the only critical value between them. 
Notation as in the preceding corollary. Under the negative gradient flow of $\ell_p$,   
$W\supset Z$ is homotopic to the disk bundle $B$ with base $Z$ 
and fiber the unit disk ${\mathbb D}^{\mu}\subset{\mathbb R}^{\mu}$, where $\mu$ is the Morse-Bott index of $Z$. So, $N_s$ is homotopic to the disk bundle $B$ restricted to
$T_s$, which is homotopic to the disk bundle $B$ restricted to $Z_s$, denoted by $B_s$. The topology of $N_s$ attached to $M_{p,a}$ is therefore
$$
H_k(B_s,\partial B_s)=H_{k-\mu}(Z_s)
$$
by Thom isomorphism. On the other hand, since $\ell_{q_j}$ assumes the critical submanifolds $C_j$ and possibly other $Y_1,\cdots,Y_a$ in $N_s$ with indexes 
$\sigma_j, \tau_1,\cdots,\tau_a$ and critical values $\alpha_j,\beta_1,\cdots,\beta_a$,
respectively, 
we see there is an isomorphism 
\begin{equation}\label{isomorphism}
H_{k-\sigma_j}(C_j)\oplus_{b=1}^a H_{k-\tau_b}(Y_b)\simeq H_{k-\mu}(Z_s),
\end{equation}
because we can find noncritical values $a'$ and $b', a'<\alpha_j,\beta_1,\cdots,\beta_a<b',$ of $\ell_{q_j}$ such that
$$
M_{p,a}\subset M_{q_j,a'}^\circ\subset M_{q_j,\alpha_j},M_{q_j,\beta_1},\cdots,M_{q_j,\beta_a},M_{p,0}\subset M_{q_j,b'}^\circ\subset M_{p,b},  
$$
so that the topology of $N_s$ attached to $M_{q_j,a'}$, which deformation retracts to $M_{p,a}$, is the left hand side of~\eqref{isomorphism}.
Shifting indexes, 
$H_k(C_j)$ embeds in $H_{k+\sigma_i-\mu}(Z)$ for all $k$.

In conclusion, if $Z_s$ is a sphere, then $C_j$ must be either a point or a sphere of the same dimension as $Z_s$. In particular, if $C_j$ is a 
curvature surface,
then $C_j$ must be a sphere of the same dimension as $Z_s$, so that the dimension of $Z_s$ is locally constant. 
\end{remark}

\begin{corollary}\label{c} Let $(q_j)$ be a sequence of focal points converging to the unit tangent cone vector $e$ at $p$. If $Z$ is
not on any level set of $l_{q_j}$ for all $j$, then $Z$ is not on any level set of $l_e$.
\end{corollary}

\begin{proof} Suppose $Z$ is on a level set of $\ell_e$. Then $Z$ must be on a regular level set of $\ell_e$; otherwise, Corollary~\ref{cor} 
would imply that $p=e$, which is not the case ($p\perp e$ on $S^n$). Therefore, as in~\eqref{gl},  
$\ell_{q_j}=\ell_p+g_j$ would be regular over $Z$ for $j\geq L$ for some $L$ when $q_j\in{\mathcal O}_e$ given in Proposition~\ref{P}. Now locally,

\begin{eqnarray}\nonumber
\aligned
\ell_p&=\sum_{j=1}^{n-1-s}\alpha_ju_j^2+O(3),\\
\ell_{e_j}&=h(u)+k(v)+\sum_{i=1}^s\delta_{i}v_i^2+\sum_{ik}\gamma_{ik}u_iv_k+O(3),
\endaligned
\end{eqnarray}
where either $h(u)$ or $k(v)$ has a nontrivial linear term, since 
$\ell_{e_j}$ is regular on $Z$. It follows that the gradient of $\ell_{q_j},j\geq L,$ is nonzero in a tubular neighborhood of $Z$, which implies that $q_j$ would not converge to $p$, a contradiction.
\end{proof}

\section{The proof} 
We do induction on $n$, the dimension of the ambient sphere, with the induction statement that all compact taut submanifolds in $S^n$ are algebraic.
The statement is clearly true for $n=1$. Assume the statement is true for $n-1$. Let us first handle a taut hypersurface $M$ in $S^n$ to show that it is 
algebraic. Since any curvature surface of $M$ is contained in a curvature sphere of dimension $n-1$, we know by the induction hypothesis that all curvature surfaces, being taut 
by Ozawa theorem, are
algebraic. We then proceed to establish that $M$ is algebraic by establishing the local finiteness property on $M$.

To this end,
let $x$ be a point in ${\mathcal G}^c$ and let $Z$
through $x$ be a curvature surface
with focal point $p$. We stipulate that all the conditions on the neck
$$  
\pi:W\rightarrow Z
$$
we encountered in the preceding section prevail.

Let $J$ be the principal index range of the principal maps~\eqref{fmap} satisfying
\begin{equation}\nonumber
p=f_a(x),\quad \forall a\in J.
\end{equation}
Since the principal multiplicities are not locally constant at $x$,
given any neighborhood $V_p$ of $p$ and $U_x$ of $x$, there is a point
$y\in(\cup_{a\in J}f_a^{-1}(V_p))\cap U_x$
and a curvature surface $C_y$ through $y$ such that
$\dim(C_y)<\dim(Z)$.
Furthermore, by the continuity of the focal maps
there are open neighborhoods $N_p$ of $p$ and $O_x$ of $x$ so small that
$f_b(O_x)$ is disjoint from $N_p$ for all $b\notin J$.
Set
\begin{equation}\label{Np}
O^*:=(\cup_{a\in J}f_a^{-1}(N_p))\cap O_x.
\end{equation}
Then for any focal point $q\in N_p$, there is a curvature surface of $\ell_q$ through $O^*$; moreover,  each curvature surfaces $C_y$ through $y\in O^*$ has  principal index range
contained in $J$ so that in particular $\dim(C_y)\leq \dim(Z)$, and moreover,
there exist $C_y$ with $\dim(C_y)<\dim(Z)$ in any neighborhood of $x$ in $O^*$ because $x\in{\mathcal G}^c$.

We further stipulate, by choosing $O^*$ so small, that

\vspace{2mm}

\noindent $Z$ be the only critical submanifold of $\ell_p$ contained in the topological closure of $W$, and,

\vspace{2mm}

\noindent any curvature surface passing through $O^*$ with principal index range contained in $J$ be entirely contained in $W$.

\vspace{2mm}

Since there are only a finite number of critical submanifolds
with the focal point $p$ and $Z$ is the only critical submanifold of $\ell_p$ in $W$,
all focal points $q$ of
curvature surfaces $\neq Z$ in $W$ are different from $p$ in any small neighborhood of $p$.

We have two cases to consider.

\vspace{2mm}

\noindent {\bf Category 1.} One of the curvature spheres of a focal point $\ell_q,q\neq p,$
contains $Z$.

\vspace{2mm}

This means that $Z$ is contained in a level set of such
a height function $\ell_q$. Suppose $Z$ is contained
in a critical submanifold of $\ell_q$.
Then by Corollary~\ref{cor}, the height functions $\ell_p$ and
$\ell_q$ share the same center of
the curvature sphere through $Z$, so that it must be that $p=q$, which is not the case.
Therefore, all points of $Z$ are regular
points of $\ell_q$.

We can understand all these $q$ explicitly. Let $S^l$ be the smallest sphere
containing $Z$. It is more convenient to view what goes on in $R^n$ when we
place the pole of the stereographic projection on $Z$.
Then we are looking at an ${\mathbb R}^l$, which, by a conformal transformation
of the sphere, we may assume is the
standard
one contained in ${\mathbb R}^n$, in which $Z$ sits. Let
$E\simeq{\mathbb R}^{n-l}$
be the orthogonal complement of the ${\mathbb R}^l$. Any ${\mathbb R}^{n-l-1}$
in $E$ gives rise to an ${\mathbb R}^{n-1}$ containing $Z$, and vice versa.
Back on the sphere, this means that
we have an $(n-l-1)$-parameter family of $S^{n-1}$ containing $Z$. The focal
points $f$ of these $S^{n-1}$ form an $S^{n-l-1}$ on the equator. 

Assume $\dim(Z)<l$ first. 
Notations as above, let us consider the incidence space
$$
{\mathcal I}\subset S^{n-l-1}\times M
$$
given by 
\begin{eqnarray}\nonumber
\aligned 
{\mathcal I}:=\{&(t,z):t\;\text{is sufficiently close to}\; p\;\text{and}\;
z\;\text{belongs to a critical}\\ 
&\text{submanifold of}\;\ell_t\;\text{passing through}\; O^*\}.
\endaligned
\end{eqnarray}
Let ${\mathcal I}^\circ\subset {\mathcal I}$ be defined by
\begin{eqnarray}\nonumber
\aligned
{\mathcal I}^\circ:=\{&(t,z)\in{\mathcal I}: z\; \text{belongs to a critical submanifold with}\\
&\text{dimension}\; <\dim(Z)\}.
\endaligned
\end{eqnarray}
Let 
$$
\Pi_2:{\mathcal I}\rightarrow M
$$ 
be the projection from ${\mathcal I}$ onto its second summand.

We may assume $O^*$ is so small that it is contained in the coordinate
chart $V$ employed in~\eqref{implicit}; we adopt the notations there. Around each $(t,z)\in {\mathcal I}^\circ$, choose a small neighborhood
$V_{t,z}\subset{\mathcal D}\times V$ over which a certain $(n-s)$-by-$(n-s)$ minor of
the Hessian matrix, given on the left hand side of~\eqref{implicit}, is nonsingular; here, $s=\dim(Z)$. Choose a countable refinement $V_1,V_2\cdots$ of the open covering
$\{V_{t,z}\}$ of ${\mathcal I}^\circ$. Fix a $V_j$, over which we may assume without loss of generality that the upper left $(n-s)$-by-$(n-s)$ minor
of the Hessian matrix is nonsingular. Via the map
$$
h:=(q,z)\in V_j:\mapsto (F_1(q,z),\cdots,F_{n-s-1}(q,z),G_1(q,z))\in {\mathbb R}^{n-s},  
$$
the implicit function theorem ensures that $h^{-1}(0)$ consists of countably many (disjoint) connected manifolds $V_{jk},k=1,2,\cdots,$ of 
dimension $n+s-1$. Each $V_{jk}$ is parametrized by $(q,v_2,\cdots,v_s)$ with the chart map
$$
g_{jk}:(q,v_2,\cdots,v_s)\in{\mathbb R}^{n+s-1}\mapsto (q,u_1,\cdots,u_{n-s-1},v_1,\cdots,v_s)\in V_{jk},
$$
where $u_1,\cdots,u_{n-s-1},v_1$ are functions of $q,v_2,\cdots,v_s$. Since $v_2,\cdots,v_s$ are coordinate functions over
$V_{jk}$ with respect to the chart, we can define the map
$$
f_{jk}:(q,z)\in V_{jk}\mapsto (v_2(q,z),\cdots,v_s(q,z))\in {\mathbb R}^{s-1}.
$$ 
Consider the map
$$
F_{jk}:V_{jk}|_{{\mathcal I}^\circ}\rightarrow {\mathbb R}^{\dim(M)-l}\times {\mathbb R}^{s-1},\quad F_{jk}=(id,f_{jk}):(t,z)\mapsto (t,f_{jk}(t,z)).
$$
It is clear that ${\mathbb R}^{\dim(M)-l}\times {\mathbb R}^{s-1}$ is of Hausdorff dimension $=\dim(M)-2$ since
$l>s$. Therefore, $V_{jk}|_{{\mathcal I}^\circ}$
is also of Hausdorff dimension at most $\dim(M)-2$ via the inverse Lipschitz-continuous
map $g_{jk}$. It follows that each $V_j|_{{\mathcal I}^\circ}$ and thus ${\mathcal I}^\circ$, and its topological closure $\overline{{\mathcal I}^\circ}$,
are of Hausdorff dimension at most $\dim(M)-2$ as well. As a consequence, $\Pi_2(\overline{{\mathcal I}^\circ})$ is of Hausdorff
dimension at most $\dim(M)-2$, which therefore does not disconnect $M$~\cite[p.\ 269]{SY}. 

In other words, the set of
points $z\in O^*$ belonging to the critical submanifolds of $\ell_t$ passing through $O^*$ with dimension $<\dim(Z)$,
for $t$ sufficiently close to $p$, does not disconnect $M$ and so does not contribute to
the local finiteness property. Once such points are excised, ${\mathcal I}\setminus \overline{{\mathcal I}^\circ}$ is a manifold of dimension
$=\dim(M)-l+\dim(Z)$, which can be seen by solving     
$$
F_1=\cdots=F_{n-1-s}=G_1=\cdots=G_s=0
$$ 
by the implicit function theorem for $u_1,\cdots,u_{n-1-s}$ in terms of
$v_1,\cdots,v_s$ and $t$. 

Lastly, observe that $\Pi_2:{\mathcal I}\setminus\overline{{\mathcal I}^\circ}\rightarrow O^*$ is a finite map, because there are only at
most $\dim(M)$ many curvature surfaces through $z\in O^*$. It is also an open map since it is the restriction to
${\mathcal I}$ of the standard projection from $S^n\times S^n$ to $S^n$. By Federer's version of Sard's theorem~\cite[p.\ 316]{Fe},
which states that the critical value set of a smooth map $f:{\mathbb R}^l\rightarrow {\mathbb R}^s$, at which the rank of the derivative is $\leq \nu$, is of 
Hausdorff $\nu$-dimensional measure zero. Consequently, the critical
value set of $\Pi_2|_{{\mathcal I}\setminus\overline{{\mathcal I}^\circ}}$ is of Hausdorff $(\dim(M)-l+\dim(Z))$-dimensional measure zero, and so 
in particular, of Hausdorff $(\dim(M)-1)$-dimensional measure zero since $l>\dim(Z)$. So by~\cite[p.\ 269]{SY}
the critical value set of $\Pi_2({\mathcal I}\setminus \overline{{\mathcal I}^\circ})$ does not disconnect $M$, which can thus be excised as well. What remains 
is thus the regular set ${\mathcal R}$ of ${\mathcal I}\setminus\overline{{\mathcal I}^\circ}$, over which $\Pi_2$ is a finite covering map onto its image. It follows that
$\Pi_2({\mathcal R})$ is an immersed 
manifold of dimension $=\dim(M)-l+\dim(Z)\leq \dim(M)-1$, which thus disconnect $M$ in only finitely many components.

If $\dim(Z)=l$, then $Z=S^l$. Remark~\ref{importantremark} implies that all curvature surfaces passing through ${\mathcal O}^*$ (by shrinking it if necessary) 
are $S^l$, so that no principal index change occurs in ${\mathcal O}^*$. This is a contradiction.


In summary, barring a closed set of Hausdorff $(\dim(M)-1)$-dimensional measure zero in $O^*$ that does not disconnect $M$, the union of the
curvature surfaces in ${\mathcal G}^c$ in Category 1
is an immersed manifold of dimension $=\dim(M)-l+\dim(Z)\leq \dim(M)-1$ and
hence disconnects ${\mathcal G}$ in at most finitely many
connected components.  

\vspace{2mm}

\noindent {\bf Category 2.} No curvature spheres of $\ell_q$ of a focal point $q,q\neq p,$ contain $Z$.

\vspace{2mm}

\begin{definition} Let ${\mathcal F}_Z$ be the set of all points $q$ for which no curvature spheres of $\ell_q$ contain $Z$, and let ${\mathcal UC}_p$ be the
set of the unit tangent cone vectors of ${\mathcal F}_Z$ at $p$.
\end{definition}

By Corollary~\ref{cor}, ${\mathcal F}_Z$ is the image of the unit normal bundle $UN$ of $Z$ under the normal exponential map
$$
Exp:((x,n),t)\in UN\times (-\pi,\pi) \mapsto \cos(t) x+\sin(t) n\in {\mathcal F}_Z.
$$
Hence, ${\mathcal F}_Z$ is semialgebraic, and so ${\mathcal UC}_p$ is semialgebraic by construction. 
Note that ${\mathcal UC}_p\subset {\mathcal F}_Z$ by Corollary~\ref{c}.

By Corollary~\ref{CCORO}, for a sequence $C_j$ of curvature surfaces through
$ y_j\in C_j$ converging to $x$, where the dimension of each $C_j$ is not locally constant, we know a unit tangent cone vector $e$ at $p$ to which a subsequence of the
focal points $q_j$ of $C_j$ converge has the property that, the dimension of the critical submanifold of $\ell_e$ through $x$ in $Z$ is not locally constant.   
Accordingly, we make the following definition. 

\begin{definition}\label{uc0}
We let ${\mathcal UC}^\circ_p\subset {\mathcal UC}_p$ be the set
where the dimension of the critical submanifold of $\ell_e|_Z$ through $x$ in $Z$ is not locally constant.
\end{definition}

\begin{lemma}\label{uc1}
${\mathcal UC}^\circ_p$ is semialgebraic.
\end{lemma}

\begin{proof} ${\mathcal UC}^\circ_p$ consists of those $e\in{\mathcal UC}_p$ for which the gradient of $\ell_e|_Z=0$ at $x$ and the kernel (or rank) of the Hessian of 
$\ell_e|_Z$ at $x$ is not locally constant. Therefore,  
Example~\ref{rnk} and Proposition~\ref{<} give the desired conclusion.
\end{proof}

The nature of ${\mathcal F}_Z$ and ${\mathcal UC}^\circ_p$  motivates us to 
look into the following semialgebraic object.   



\begin{lemma}\label{cOR}
The set
$UN^o$ of unit normals $\xi$ of
$Z$ at which the shape operator $S_\xi$ has multiplicity change
is semialgebraic of dimension $\leq \dim(M)-1$.
\end{lemma}

\begin{proof} Let $\dim(Z)=s$ and let $(y,\zeta)\in Z\times S^{n-s-1}$
parametrize the unit normal bundle $UN$ of $Z$.
The characteristic polynomial of $S_\xi$ is of the form
$$
\lambda^s+a_{s-1}\lambda^{s-1}+\cdots+a_1\lambda+a_0,
$$
where $a_1,\cdots,a_{s-1}$ are polynomials in the zero jet of $\zeta$ and the second
jets of $y$; hence they are Nash functions. By the discussion following Theorem~\ref{slice} (the slicing theorem),
$Z\times S^{n-s-1}$ is decomposed into finitely many disjoint semialgebraic sets
$V_1,\cdots,V_\tau$, where each $V_i$ is equipped with semialgebraic functions
$\eta_{i,1}<\cdots<\eta_{i,{l_{i}}}$ that solve the characteristic polynomial, counting multiplicities; moreover, 
$UN^o$, where the principal multiplicities
are not locally constant, is semialgebraic of a lower dimension by Proposition~\ref{<} and the discussion preceding it. So 
\begin{equation}\label{M}
\dim(UN^o)\leq n-2=\dim(M)-1.
\end{equation}
\end{proof}


Now in view of
Corollary~\ref{cor}, for a unit normal $\xi$ to $Z$, we let
$q^1_\xi,q^2_\xi,\cdots,$ and $q^{\dim(Z)}_\xi$ be the
focal point of the curvature surface
through the base point of $\xi$ corresponding to the principal curvature function
$\lambda^1(\xi),\cdots$, and $\lambda^{\dim(Z)}(\xi)$ of $S_\xi$, respectively. The
remark following Corollary~\ref{cor} gives
the focal maps $g^1,\cdots,g^{\dim(Z)}$ that send $\xi$ to the respective focal points;
by the algebraic nature of $Z$, all these maps are semialgebraic.

Consider the semialgebraic set
${\mathcal X}\subset UN^o\times S^n\times S^n$ defined by
$$
{\mathcal X}:=\{(\xi,q,r):q=g^j(\xi)\,\text{for some}\, j; r\,
\text{belongs a critical set of}\, Z\, \text{of}\;\ell_q\}.
$$
Due to the nature of all these defining functions, ${\mathcal X}$ is
semialgebraic. (For instance, critical submanifolds are obtained by setting
the first derivative of the height function equal to zero on $Z$, which
is a semialgebraic process.)
Let
\begin{equation}\label{proj}
pr:UN\times S^n\times S^n\rightarrow S^n\times S^n
\end{equation}
be the standard
projection, which is a Nash submanifold, and let 
$$
{\mathcal J}:=pr({\mathcal X}).
$$
The set ${\mathcal J}$ is also semialgebraic.

\begin{lemma}\label{j}
$\dim({\mathcal J})\leq\dim(M)-1$.
\end{lemma}

\begin{proof} 
Consider
\begin{equation}\label{alpha}
\alpha:{\mathcal J}\rightarrow UN^o,\quad \alpha:(q,z)\mapsto \xi(q,z),
\end{equation}
where $\xi(q,z)$ is the unit tangent vector, based at $z$, along the geodesic of $S^n$ from $q$ to $z$. Note that 
$\alpha$ is the restriction to ${\mathcal J}$ of the Nash map
$$
\beta:S^n\times S^n\rightarrow S^n\times S^n,\quad \beta:(u,v)\mapsto \xi(u,v). 
$$
Note also that
$\alpha$ is a finite map, since in general each $\xi\in (UN^o)_z$ gives rise to at most $\dim(Z)$
many curvature surfaces through $z$. 
Therefore, by~\eqref{M} and~\eqref{alpha}, we have
\begin{equation}\label{dim}
\dim({\mathcal J})\leq\dim(UN^o)\leq \dim(M)-1.
\end{equation}
\end{proof}

Now we let
\begin{equation}\label{projection}
\Pi^1,\Pi^2:S^n\times S^n\rightarrow S^n
\end{equation}
be the standard projections onto the first and second summands,
respectively. Note that 
$\Pi^2$ is a finite map because through each point
in $M$ there are only at most $\dim(M)$ many critical submanifolds. Moreover,
$$
{\mathcal UC}^\circ_p\subset\Pi^1({\mathcal J})
$$
by construction.

\begin{corollary} The set
$$
{\mathcal I}:=(\Pi^1|_{\mathcal J})^{-1}({\mathcal UC}^\circ_p)
$$
is semialgebraic of dimension $\leq\dim(M)-2$.
\end{corollary}

\begin{proof} The dimension of ${\mathcal I}$ is 1 less than $\dim({\mathcal J})$ given in Lemma~\ref{j} because ${\mathcal UC}^\circ_p$ consists of unit tangent cone 
vectors. (This can be seen most clearly in ${\mathbb R}^n$ in place of $S^n$.)
\end{proof}

Recall the open set ${\mathcal O}_e$ defined before~\eqref{elll},
which is semialgebraic. We now stipulate that ${\mathcal O}_e\subset N_p$
defined in~\eqref{Np} and set
$$
{\mathcal O}:=\cup_{e\in{\mathcal UC}^\circ_p}{\mathcal O}_e.
$$




\begin{corollary}
${\mathcal O}\cap{\mathcal F}_Z$
is $\sigma$-semialgebraic in the sense that it is a 
countable union of increasing compact semialgebraic sets 
$X_1\subset X_2\subset X_3\subset\cdots$, because there is a compact exhaustion of ${\mathcal O}$.
\end{corollary}

In view of Proposition~\ref{P}, we let
$$
{\mathcal K}\subset({\mathcal O}\cap {\mathcal F}_Z)\times M\subset S^n\times S^n
$$
be the incidence space
\begin{eqnarray}\nonumber
\aligned
{\mathcal K}:=\{&(q,z):\text{for}\;q\in{\mathcal O}_e,z\in\;\text{a critical submanifold of}\;\ell_q\subset\;\text{a neck}\\
&N_1\supset\;\text{the critical submanifold}\; Z_1\;\text{of}\;\ell_e|_Z,x\in Z_1,\;\text{as given in}\\
&\text{Proposition}~\ref{P}\}.
\endaligned
\end{eqnarray}

\begin{proposition}
Away from a closed subset\, ${\mathcal N}$ of Hausdorff $(\dim(M)-1)$-measure zero, $\Pi^2({\mathcal K})$ is a manifold of dimension at most $\dim(M)-1$.
\end{proposition}

\begin{proof} 



Let
$X_1\subset X_2,\subset X_3,\cdots$ be a countable collection of increasing compact
semialgebraic sets whose union is
${\mathcal O}\cap {\mathcal F}_Z$.
By Proposition~\ref{tri}, fix a semialgebraic open cell
decomposition ${\mathcal T}_j$ of $X_j, j=1,2,\cdots,$ in such a way that ${\mathcal T}_j$
is a sub-decomposition of ${\mathcal T}_{j+1}$ for all $j$ (by decomposing $X_{j+1}\setminus X_j$);
let $F_k, k=1,2,\cdots,s,$ be the open cells in the decomposition.

We now collect the unit tangent cone vectors of $F_k$ at $p$ and call the set $UC_k$ (it is empty if $p$ is not in the closure of $F_k$ ), which is 
semialgebraic with
\begin{equation}\label{uck}
\dim(UC_k)+1=\dim(F_k)
\end{equation} 
(because $UC_k$ consists of unit tangent cone vectors) if $UC_k$ is not empty.  

\begin{sublemma} 
$$
{\mathcal UC}^\circ_p=\cup_k UC_k.
$$
\end{sublemma}

\begin{proof} ${\mathcal UC}^\circ_p$ is a closed set since those $e$
for which the dimension of the critical submanifold of $\ell_e$ is
locally constant constitute an open set. Thus for a sequence $q_j$ of $F_k$ converging to the unit tangent cone vector $e$ at $p$, let $q_j\in O_{e_j}$
for some $e_j\in{\mathcal UC}^\circ_p$. Then a converging subsequence of $e_j$ must converge to $e\in{\mathcal UC}^\circ_p$. So, 
$\cup_k UC_k\subset{\mathcal UC}^\circ_p$.

Conversely, each ${\mathcal O}_e$ contains a sequence $q_j$ of points in ${\mathcal F}_Z$ converging to $e$ at $p$. Now
choose a small compact semialgebraic disk ${\mathcal B}$ around the focal point $p$ of $Z$.
${\mathcal B}$ intersects only finitely many compact $X_1\subset X_2\subset\cdots$, and so finitely many $F_1,\cdots,F_b$. Thus, there is a subsequence of 
$q_j$ falling in one of 
these $F_1,\cdots,F_b$, say, $F_1$; it follows that $e$ is a unit tangent cone vector of $F_1$ at $p$. That is, $\cup_k U_k\supset{\mathcal UC}^\circ_p$. 
\end{proof}



Define the semialgebraic set
$$
Q_k:=(\Pi^1|_{\mathcal J})^{-1}(UC_k)\subset{\mathcal I};
$$
clearly, we have 
$$
\dim(Q_k)\leq\dim({\mathcal I})\leq\dim(M)-2.
$$
Lastly, let
$$
P_k:=(\Pi^1|_{\mathcal K})^{-1}(F_k).
$$


\begin{sublemma} 
Assume $UC_k$ is not empty. $P_k$ is then of Hausdorff\; $({\dim(M)-1})$-measure zero provided\; $\dim(Q_k)\leq \dim(M)-3$.
\end{sublemma}

\begin{proof} Since each $F_k$ is an open cube with $p$ on its boundary, we can set up a Nash diffeomorphism $\iota_k$ between a cone extended out of $UC_k$, denoted
by ${\mathbb R}^+UC_k$, and $F_k$,
$$
\iota_k:{\mathcal R}^+UC_k\rightarrow F_k,
$$
such that $\iota_k(te),e\in UC_k,0<t<\alpha_e$ for some $\alpha_e,$ all lie in ${\mathcal O}_e$. (Again, this is most clearly seen when viewed in ${\mathbb R}^n$
in place of $S^n$, where $\iota_k(te)$ is the axis of the cone ${\mathcal O}_e$.) Note that, in particular,
the fiber over $\iota_k(te),0<t<\alpha_e,$ is of dimension $\leq$
the dimension of the fiber over $e$ by Proposition~\ref{P}. Thus, by
an analysis analogous to the one in Category 1, we obtain, by~\eqref{uck},
$$
\text{Hausdorff}\;\dim(P_k)\leq\dim(Q_k)+1\leq\dim(M)-2.
$$

The countable union of all $P_k$, where $\dim(Q_k)\leq \dim(M)-3$, is thus also of Hausdorff $(\dim(M)-1)$-measure
zero.
\end{proof}


We let 
$$
{\mathcal N}_1:=\cup_k \Pi^2(P_k),\quad \dim(Q_k)\leq \dim(M)-3.
$$ 
${\mathcal N}_1$ is of Hausdorff $(\dim(M)-1)$-measure zero
so that ${\mathcal N}_1$ does not
disconnect $M$~\cite[p.\ 269]{SY}, and so does not contribute to
the local finiteness property.

Thus what is left now are $P_k$ with 
$$\dim(Q_k)=\dim(M)-2
$$ 
(assuming $UC_k$ is not empty). Note that in the 
fibration of $P_k$ 
over $F_k$ 
we may ignore the fibers over those $\iota_k(te)$ 
of fiber dimension less than the fiber dimension $d_e$ over $e$ in the fibration of $Q_k$ over $UC_k$.  In fact,
let $C$ be the fiber over $\iota_k(te)$ of dimension
$d<d_e$.
Then the Hessian matrix on the left hand side of~\eqref{implicit} is of kernel dimension
$d$ on $C$, so that around $\iota_k(te)$ in $F_k$ there is a neighborhood $N$  
over which the Hessian matrix is of kernel dimension $\leq d$ around $C$,
such that the neighboring critical submanifolds around $C$ in the fibration of $P_k$ are
contained in
a smooth family $P_k^*$ parametrized by $q\in N,$ each of whose fibers is of
dimension $=d<d_e$.
Therefore, $P_k^*$ is a manifold of dimension $\leq \dim(M)-2$ and
so $P^*_k$ and its topological closure $\overline{P^*_k}$ are of Hausdorff $(\dim(M)-1)$-measure zero.
We let 
$$
{\mathcal N}_2:=\;\text{the union of such}\; \Pi^2(\overline{P_k^*}).
$$

We denote by $P_k^\circ$ the remaining part of $P_k$ away from the preceding two classes of sets we excised. 
$P_k^\circ$ is then a smooth manifold of dimension $\dim(M)-1$.
By Federer's version of
Sard's theorem~\cite[p.\ 316]{Fe},
the critical value set ${\mathcal C}_k\subset M$ of 
$$
\Pi^2:P_k^\circ\rightarrow M
$$
is of rank $\leq\dim(M)-1$ and so is of Hausdorff
$(\dim(M)-1)$-measure zero.
This implies that  
$C_k$ and hence the union of all $C_k$
do not disconnect
$M$. We let 
$$
{\mathcal N}_3=\cup_k {\mathcal C}_k.
$$

Now let 
${\mathcal N}$ be the topological closure of the union of ${\mathcal N}_1$, ${\mathcal N}_2$, ${\mathcal N}_3$; 
${\mathcal N}$ is of Hausdorff $(\dim(M)-1)$-measure zero.
In particular, ${\mathcal N}$ is closed 
and does not disconnect $M$ and so is nowhere dense in $M$. Hence,
\begin{equation}\label{pijk}
\Pi^2:P_k\setminus (\Pi^2)^{-1}({\mathcal N})\rightarrow S^n\setminus {\mathcal N}
\end{equation}
is a locally diffeomorphic finite map from a manifold of dimension $\dim(M)-1$ into $M$, whose image
thus consists of at most $\beta(M)$ immersed manifolds of dimension $\dim(M)-1$ in $M$.

Lastly, we need to take care of those $F_k$ for which $UC_k$ are empty, i.e., whose closures do not contain $p$. 
Choose a small compact semialgebraic disk ${\mathcal B}$ around the focal point $p$.
${\mathcal B}$ intersects only finitely many $X_1\subset X_2\subset\cdots$, and hence finitely many
$F_1,\cdots,F_b$, from which we remove those whose closures do not contain $p$ and call the remaining ones $F_{i_1},\cdots,F_{i_a}$.
We then go through the same arguments as above for each of $F_{i_1},\cdots,F_{i_a}$ to conclude that $\Pi^2(P_{i_1}),\cdots,\Pi^2(P_{i_a})$ 
are immersed submanifolds of codimension 1 in the neighborhood ${\mathcal B}$ around $p$, away from a closed set that does not disconnect $M$.
\end{proof} 

In summary, barring a closed set of Hausdorff $(\dim(M)-1)$-dimensional measure zero in $W$ that does not disconnect $M$, the union of the
curvature surfaces in ${\mathcal G}^c$ in Category 2
is locally an immersed manifold of dimension $\leq \dim(M)-1$, and hence
locally disconnects ${\mathcal G}$ in at most finitely many
connected components.

\vspace{2mm}

We have established the local finiteness property in the hypersurface case.

\vspace{2mm}

We now handle the case when $M$ is a taut submanifold.
It is more convenient to work in ${\mathbb R}^n$. Let
$M_\epsilon$ be a tube over $M$ of 
sufficiently small
radius that $M_\epsilon$ is an embedded hypersurface in ${\mathbb R}^n$.  Then
$M_\epsilon$ is a taut hypersurface~\cite{Pi},
so that by the above $M_\epsilon$ is algebraic.  
Consider the focal map $F_\epsilon : M_\epsilon \rightarrow M \subset {\mathbb R}^n$ given
by
\begin{equation}\label{dis}
F_\epsilon(x) = x - \epsilon \xi,
\end{equation}
where $\xi$ is the outward field of unit normals to the tube $M_\epsilon$.  Any point of $M_\epsilon$ has an open
neighborhood $U$ parametrized by an analytic algebraic map.  The first derivatives
of this parametrization are also analytic algebraic~\cite[p.\ 54]{BCR}, and thus the Gram-Schmidt
process, applied to these first derivatives and some constant non-tangential vector, produces the
vector field $\xi$ and shows
that $\xi$ is analytic algebraic on $U$.  Hence $F_\epsilon$ is analytic algebraic on $U$ and so the image 
$F_\epsilon (U) \subset M$ is a semialgebraic subset of ${\mathbb R}^n$.  Covering $M_\epsilon$ by finitely many sets
of the kind of $U$, we see that $M$, being the union of their images under
$F_\epsilon$, is a semialgebraic subset
of ${\mathbb R}^n$. Then the Zariski closure $\overline{M}^{\text{zar}}$ of $M$ is
an irreducible algebraic variety of the same dimension as $M$ and contains
$M$ (see~\cite{CCJ} for more details).  

The induction is thus completed.


\thebibliography{100}
\bibitem{Ba} B.\ Bhata, {Matrix Analysis}, Graduate Texts in Mathematics, Vol. 196,
Springer, New York, 1996.
\bibitem{BCR} J.\ Bochnak, M.\ Coste and M.-F. Roy, {Real Algebraic Geometry},
vol. 36, Ergebnisse der Mathematik und ihrer Grenzgebiete, Springer, Berlin, 1998.
\bibitem{Bt} R.\ Bott, {\em Nondegenerate critical manifolds}, {Ann.\ Math.}\,
{\bf 60}(1954), 248-261.
\bibitem{CCJ} T.\ Cecil, Q.-S.\ Chi and G.\ Jensen, {\em On Kuiper's question
whether taut submanifolds are algebraic}, {Pacific J.\ Math.}\,
{\bf 234}(2008), 229-247.
\bibitem{Ce} T.\ Cecil, {Taut and Dupin submanifolds}, pp. 135-180, in Tight and Taut Submanifolds,
edited by T.\ Cecil and S.-S. Chern, Mathematical Sciences Research Institute
Publications, vol. 32, 1997.
\bibitem{Fe} H. Federer, {Geometric Measure Theory}, Die Grundlehren der
mathematischen
Wisenschaften {\bf 153}, Springer, New York, 1969.
\bibitem{Ku1} N.\ H.\ Kuiper, {\em Tight embeddings and maps. Submanifolds of
geometrical class three in $E^n$}, {The Chern Symposium 1979 (Proc. Internat. Sympos.,
Berkeley, Calif., 1979)}, 97-145, Springer-Verlag,
Berlin, Heidelberg, New York, 1980.
\bibitem{Ku} N.\ H.\ Kuiper, {\em Geometry in total absolute curvature theory},
pp. 377-392, in Perspectives in Mathematics (Oberwolfach 1984). 
\bibitem{MB} R.\ Bott, {\em Nondegenerate critical manifolds}, {Ann.\ Math.}\, {\bf 60}(1954), 248-261.
\bibitem{Oz} T.\ Ozawa, {\em On critical sets of distance functions to a taut
submanifold}, {Math.\ Ann.}\, {\bf 276}(1986), 91-96.
\bibitem{Ma} R.\ Miyaoka, {\em Taut embeddings and Dupin hypersurfaces}, 
Differential Geometry in Submanifolds, Lect. Notes in Math. {\bf 1090} (1984) 15-23.
\bibitem{Pi} U.\ Pinkall, {\em Curvature properties of taut submanifolds},
{Geom. Dedicata}\, {\bf 20}(1986), 79-83.
\bibitem{Re} H.\ Reckziegel, {\em On the eigenvalues of the shape operator of an
isometric immersion into a space of constant curvature}, {Math.\ Ann.}\, {\bf 243}(1979), 71-82.
\bibitem{SY} R.\ Schoen and S.\-T.\ Yau, {Lectures on differential geometry},
Conference Proceedings and Lecture Notes in Geometry and Topology, I,
International Press, Cambridge, MA, 1994.
 \bibitem{TT} C.-L.\ Terng and G.\ Thorbergsson, {\em Submanifold geometry in symmetric spaces}, {J.\ Differ.\ Geom.}\, {\bf 42}(1995), 665-718.                                                                            
\bibitem{TT1} C.-L.\ Terng and G.\ Thorbergsson, {\em Taut immersions into complete Riemannian manifolds}, {Tight and Taut Immersions}, MSRI Publications {\bf 32}(1997), 181-228.
\end{document}